\documentclass[reqno, 11pt]{amsart}

\usepackage[latin5]{inputenc}

\newtheorem{thm}{Theorem}

\theoremstyle{definition}
\newtheorem{defn}[thm]{Definition}
\theoremstyle{remark}

\numberwithin{equation}{section}

\begin{document}

\title[Bivariate Cheney-Sharma operators ]{Bivariate Cheney-Sharma  operators on  simplex }%

\author[ G\"{U}LEN BA\c{S}CANBAZ-TUNCA,  AY\c{S}EG\"{U}L EREN\c{C}\.{I}N, HATİCE G\"{U}L \.{I}NCE-\.{I}LARSLAN ]
{G\"{U}LEN BA\c{S}CANBAZ-TUNCA*,
 AY\c{S}EG\"{U}L EREN\c{C}\.{I}N**, HATİCE G\"{U}L \.{I}NCE-\.{I}LARSLAN***}

\newcommand{\acr}{\newline\indent}
\address{\llap{*\,}Ankara University\acr
Faculty of Science\acr Department of Mathematics\acr 06100,
Tando\v{g}an, Ankara, Turkey} \email{tunca@science.ankara.edu.tr}
\address{\llap{**\,}Abant \.{I}zzet Baysal University\acr
Faculty of Arts and Science\acr Department of Mathematics\acr 14280,
Bolu, Turkey} \email{erencina@hotmail.com}
\address{\llap{***\,}Gazi University\acr
Faculty of Arts and Science\acr Department of Mathematics\acr 06500,
Teknikokullar, Ankara, Turkey} \email{ince@gazi.edu.tr}


%s}%

% ----------------------------------------------------------------
\begin{abstract}
 In this paper, we consider bivariate Cheney-Sharma operators which are not the
tensor product construction. Precisely, we show that these operators
 preserve Lipschitz condition of a given Lipschitz continuous
function  $f$ and also the properties of the  modulus of continuity
function when $f$ is a modulus of continuity function.
\end{abstract}
% ----------------------------------------------------------------
\maketitle
% ----------------------------------------------------------------
\section{\textbf{Introduction}}
The most celebrated linear positive operators for the uniform
approximation of continuous real valued functions on $[0,1]$ are
Bernstein polynomials. As it is well known, besides approximation
results, Bernstein polynomials have some nice retaining properties.
The most referred study in this direction was due to Brown, Elliott
and Paget  \cite{bep} where they gave an elementary proof for the
preservation of the Lipschitz constant and order of a Lipschitz
continuous function by the Bernstein polynomials. Whereas, Lindvall
previously obtained this result in terms of probabilistic methods in
\cite{lin}. Moreover, in \cite{li} Li proved that Bernstein
polynomials also preserve the properties of the  function of modulus
of continuity. The same problems for some other type univariate or
multivariate linear positive operators were solved by either using
an elementary or probabilistic way (see, e.g. \cite{gy}-\cite{gcs},
\cite{fc}, \cite{cao}, \cite{agf}, \cite{agf1}, \cite{kh},
\cite{khp}, \cite{ti}).

 In Abel-Jensen identity (see
\cite{al}, p.326)
\begin{equation}\label{eq:1.1}
(u+v)\left(u+v+m\beta\right)^{m-1}=\sum^{m}_{k=0}\binom{m}{k}u\left(u+k\beta\right)^{k-1}v\left[v+(m-k)\beta\right]^{m-k-1}
\end{equation}
where $u$, $v$ and $\beta \in \mathbb{R }$, by taking $u=x$, $v=1-x$
and $m=n$, Cheney-Sharma \cite{cs} introduced the following
Bernstein type operators for $f\in C[0,1]$, $x\in[0,1]$ and
$n\in\mathbb{N}$

\begin{equation*}
Q^{\beta}_{n}\left(f;x\right)=\left(1+n\beta\right)^{1-n}\sum^{n}_{k=0}f\left(\frac{k}{n}\right)
\binom{n}{k}x\left(x+k\beta\right)^{k-1}(1-x)\big[1-x+(n-k)\beta\big]^{n-k-1},
\end{equation*}
where $\beta$ is a nonnegative real parameter. For these operators,
 tensor product of them  and their some generalizations
we can cite the papers \cite{ar}, \cite{ca}, \cite{c},
\cite{sc}-\cite{ta} and the monograph \cite{al}. Remark that from
\cite{cs} and \cite{sc} we know that $Q^{\beta}_{n}$ operators
reproduce constant functions and linear functions. Very recently, in
\cite{gcs} the authors showed that univariate Cheney-Sharma
operators  preserve the Lipschitz constant and order of a Lipschitz
continuous function as well as the properties of the function of
modulus of continuity.

We now introduce the notations, some needful definitions and the
construction of the bivariate operators.
\\ Throughout the paper, we
shall use the standard notations given
below. \\
Let $\mathbf x=(x_{1},x_{2})\in\mathbb{R}^{2}$, $\mathbf
k=(k_{1},k_{2})\in \mathbb{N}_{0}^{2}$ , $\mathbf e=(1,1)$, $\mathbf
0=(0,0)$ ,$0\leq\beta\in\mathbb{R}$ and $n\in\mathbb{N}$. We denote
as usual
$$|\mathbf x|:=x_{1}+x_{2},\quad\mathbf x^{\mathbf k}:=x^{k_{1}}_{1}x^{k_{2}}_{2},\quad |\mathbf k|:=k_{1}+k_{2},\quad\mathbf k\mathbf!:=k_{1}!k_{2}!
,\quad \beta\mathbf x=(\beta x_{1},\beta x_{2})$$ and
$$\binom{n}{\mathbf k}:=\frac{n!}{\mathbf k\mathbf!(n-|\mathbf k|)!}, \quad  \sum_{|\mathbf k|\leq n}:=\sum^{n}_{k_{1}=0}\sum^{n-k_{1}}_{k_{2}=0}.$$
We also denote  the two dimensional simplex
 by
$$S:=\left\{\mathbf x=(x_{1},x_{2})\in\mathbb{R}^{2}:x_{1},x_{2}\geq 0, |\mathbf x|\leq 1\right\}.$$
Moreover, $\mathbf x\leq \mathbf y$ stands for $x_{i}\leq y_{i}$,
$i=1,2$.

We now construct the non-tensor product Cheney-Sharma operators.
From (\ref{eq:1.1}) it is clear that
\begin{equation*}
\left(1+n\beta\right)^{n-1}=\sum^{n}_{k_{1}=0}\binom{n}{k_{1}}x_{1}\left(x_{1}+k_{1}\beta\right)^{k_{1}-1}(1-x_{1})\left[1-x_{1}+(n-k_{1})\beta\right]^{n-k_{1}-1}.
\end{equation*}
In (\ref{eq:1.1}), taking $u=x_{2}$, $v=1-x_{1}-x_{2}$ and
$m=n-k_{1}$ we have
\begin{equation*}
\begin{split}
(1-x_{1})\left[1-x_{1}+(n-k_{1})\beta\right]^{n-k_{1}-1}=&\sum^{n-k_{1}}_{k_{2}=0}\binom{n-k_{1}}{k_{2}}x_{2}\left(x_{2}+k_{2}\beta\right)^{k_{2}-1}(1-x_{1}-x_{2})\\
&\times\left[1-x_{1}-x_{2}+(n-k_{1}-k_{2})\beta\right]^{n-k_{1}-k_{2}-1}.
\end{split}
\end{equation*}
Using this result in the above equality we find that
\begin{equation*}
\begin{split}
\left(1+n\beta\right)^{n-1}=&\sum^{n}_{k_{1}=0}\sum^{n-k_{1}}_{k_{2}=0}\binom{n}{k_{1}}\binom{n-k_{1}}{k_{2}}x_{1}x_{2}
\left(x_{1}+k_{1}\beta\right)^{k_{1}-1}\left(x_{2}+k_{2}\beta\right)^{k_{2}-1}\\
&\times(1-x_{1}-x_{2})\left[1-x_{1}-x_{2}+(n-k_{1}-k_{2})\beta\right]^{n-k_{1}-k_{2}-1}\\
=&\sum_{|\mathbf k|\leq n}\binom{n}{\mathbf k}\mathbf x^{\mathbf
e}(\mathbf x+\mathbf k \beta)^{\mathbf k-\mathbf e}\left(1-|\mathbf
x|\right)\left[1-|\mathbf x|+(n-|\mathbf k|)\beta\right]^{n-|\mathbf
k|-1}.
\end{split}
\end{equation*}
or
\begin{equation*}
1=\left(1+n\beta\right)^{1-n} \sum_{|\mathbf k|\leq
n}\binom{n}{\mathbf k}\mathbf x^{\mathbf e}(\mathbf x+\mathbf k
\beta)^{\mathbf k-\mathbf e}\left(1-|\mathbf
x|\right)\left[1-|\mathbf x|+(n-|\mathbf k|)\beta\right]^{n-|\mathbf
k|-1}.
\end{equation*}

 In this paper, for a continuous real valued function $f$ defined
on $S$ we consider the non-tensor product bivariate extension of the
operators $Q^{\beta}_{n}(f;x)$ defined by
\begin{equation}\label{eq:2.1}
\begin{split}
G^{\beta}_{n}(f;\mathbf
x)=&\left(1+n\beta\right)^{1-n}\sum_{|\mathbf k|\leq
n}f\left(\frac{\mathbf k}{n}\right)\binom{n}{\mathbf k}\mathbf
x^{\mathbf e}(\mathbf x+\mathbf k
\beta)^{\mathbf k-\mathbf e}\\
&\times\left(1-|\mathbf x|\right)\left[1-|\mathbf x|+(n-|\mathbf
k|)\beta\right]^{n-|\mathbf k|-1}
\end{split}
\end{equation}
where $\beta$ is a nonnegative real parameter, $\mathbf x\in S$ and
$n\in\mathbb{N}$. We observe that for $\beta=0$ these operators
reduce to non-tensor product bivariate Bernstein polynomials (see
\cite{df},\cite{fa}).
\begin{defn}(see, e.g.\cite{cao})
A continuous  real valued function $f$ defined on $A\subseteq
\mathbb{R}^{2}$
 is said to be Lipschitz continuous function of
order $\mu$ , $0<\mu\leq 1$ on $A$, if there exists $M> 0$ such that

$$\left|f(\mathbf x)-f(\mathbf y)\right|\leq M \sum^{2}_{i=1}\left|x_{i}-y_{i}\right|^{\mu}$$
for all $\mathbf x,\mathbf y\in A$. The set of Lipschitz continuous
functions of order $\mu$ with Lipschitz constant $M$ on $A$ is
denoted by
$Lip_{M}(\mu,A)$.\\
\end{defn}

\begin{defn}(see, e.g.\cite{fc})
If a bivariate nonnegative and continuous function $\omega(\mathbf
u)$ satisfies the following conditions,
then it is called a function of  modulus of continuity.\\
 $(a)$ $\omega(\mathbf 0)=0$,\\
 $(b)$ $\omega(\mathbf u)$ is a non-decreasing function in $\mathbf
 u$, i.e., $\omega(\mathbf u)\geq\omega(\mathbf v)$ for $\mathbf u\geq\mathbf
 v$,\\
 $(c)$ $\omega(\mathbf u)$ is semi-additive, i.e., $\omega(\mathbf u+\mathbf v)\leq\omega(\mathbf u)+\omega(\mathbf
 v)$.
\end{defn}

\section{\textbf{Main results }}
In this section,  inspired by the paper of Cao, Ding and Xu
\cite{cao}, including preservation properties of multivariate
Baskakov operators, we show that  non-tensor product Cheney-Sharma
operators defined by $G^{\beta}_{n}$  preserve the Lipschitz
condition of a given Lipschitz continuous function $f$ and
properties of the function of modulus of continuity when the
attached function $f$ is a modulus of continuity function.

\begin{thm}
If $f\in Lip_{M}(\mu,S)$, then $G^{\beta}_{n}(f)\in Lip_{M}(\mu,S)$
for all $n \in\mathbb{N}$.
\end{thm}
\begin{proof}
Let $\mathbf x,\mathbf y\in S$ such that $\mathbf y\geq \mathbf x$.
From (\ref{eq:2.1}) we have
\begin{equation*}
\begin{split}
G^{\beta}_{n}(f;\mathbf
y)=&\left(1+n\beta\right)^{1-n}\sum^{n}_{i_{1}=0}\sum^{n-i_{1}}_{i_{2}=0}f\left(\frac{\mathbf
i}{n}\right)\binom{n}{\mathbf i}\mathbf y^{\mathbf e}\left(\mathbf
y+\mathbf i\beta\right)^{\mathbf i-\mathbf e}\\
&\times\left(1-|\mathbf y|\right)\left[1-|\mathbf y|+(n-|\mathbf
i|)\beta\right]^{n-|\mathbf i|-1}\\
=&\left(1+n\beta\right)^{1-n}\sum^{n}_{i_{1}=0}\sum^{n-i_{1}}_{i_{2}=0}f\left(\frac{\mathbf
i}{n}\right)\binom{n}{\mathbf
i}y_{1}\left(y_{1}+i_{1}\beta\right)^{i_{1}-1}\\
&\times y_{2}\left(y_{2}+i_{2}\beta\right)^{i_{2}-1}\left(1-|\mathbf
y|\right)\left[1-|\mathbf y|+(n-|\mathbf i|)\beta\right]^{n-|\mathbf
i|-1}.
\end{split}
\end{equation*}
Setting $u=x_{1}$, $v=y_{1}-x_{1}$, $m=i_{1}$ and  $u=x_{2}$,
$v=y_{2}-x_{2}$, $m=i_{2}$, respectively, in (\ref{eq:1.1}), we find
\begin{equation*}
y_{1}\left(y_{1}+i_{1}\beta\right)^{i_{1}-1}=\sum^{i_{1}}_{k_{1}=0}\binom{i_{1}}{k_{1}}x_{1}\left(x_{1}+k_{1}\beta\right)^{k_{1}-1}\left(y_{1}-x_{1}\right)
\left[y_{1}-x_{1}+(i_{1}-k_{1})\beta\right]^{i_{1}-k_{1}-1}
\end{equation*}
and
\begin{equation*}
y_{2}\left(y_{2}+i_{2}\beta\right)^{i_{2}-1}=\sum^{i_{2}}_{k_{2}=0}\binom{i_{2}}{k_{2}}x_{2}\left(x_{2}+k_{2}\beta\right)^{k_{2}-1}\left(y_{2}-x_{2}\right)
\left[y_{2}-x_{2}+(i_{2}-k_{2})\beta\right]^{i_{2}-k_{2}-1}.
\end{equation*}
Therefore,
\begin{equation*}
\begin{split}
G^{\beta}_{n}(f;\mathbf
y)=&\left(1+n\beta\right)^{1-n}\sum^{n}_{i_{1}=0}\sum^{n-i_{1}}_{i_{2}=0}f\left(\frac{\mathbf
i}{n}\right)\binom{n}{\mathbf
i}\sum^{i_{1}}_{k_{1}=0}\sum^{i_{2}}_{k_{2}=0}\binom{i_{1}}{k_{1}}\binom{i_{2}}{k_{2}}x_{1}x_{2}\\
&\times\left(x_{1}+k_{1}\beta\right)^{k_{1}-1}\left(x_{2}+k_{2}\beta\right)^{k_{2}-1}\left(y_{1}-x_{1}\right)\left(y_{2}-x_{2}\right)\\
&\times\left[y_{1}-x_{1}+(i_{1}-k_{1})\beta\right]^{i_{1}-k_{1}-1}\left[y_{2}-x_{2}+(i_{2}-k_{2})\beta\right]^{i_{2}-k_{2}-1}\\
&\times\left(1-|\mathbf y|\right)\left[1-|\mathbf y|+(n-|\mathbf
i|)\beta\right]^{n-|\mathbf i|-1}\\
=&\left(1+n\beta\right)^{1-n}\sum^{n}_{i_{1}=0}\sum^{i_{1}}_{k_{1}=0}\sum^{n-i_{1}}_{i_{2}=0}\sum^{i_{2}}_{k_{2}=0}f\left(\frac{\mathbf
i}{n}\right)\frac{n!}{\mathbf k!(n-|\mathbf
i|)!(i_{1}-k_{1})!(i_{2}-k_{2})!}\\
&\times\mathbf x^{\mathbf e}(\mathbf x+\mathbf k \beta)^{\mathbf
k-\mathbf e}(\mathbf y-\mathbf x)^{\mathbf e}\left[\mathbf y-\mathbf
x+(\mathbf i-\mathbf k)\beta\right]^{\mathbf i-\mathbf
k-\mathbf e}\\
&\times\left(1-|\mathbf y|\right)\left[1-|\mathbf y|+(n-|\mathbf
i|)\beta\right]^{n-|\mathbf i|-1}.
\end{split}
\end{equation*}
Changing the order of the above summations and then letting $\mathbf
i-\mathbf k=\mathbf l$ we obtain
\begin{equation}\label{eq:3.1}
\begin{split}
G^{\beta}_{n}(f;\mathbf
y)=&\left(1+n\beta\right)^{1-n}\sum^{n}_{k_{1}=0}\sum^{n}_{i_{1}=k_{1}}\sum^{n-i_{1}}_{k_{2}=0}\sum^{n-i_{1}}_{i_{2}=k_{2}}f\left(\frac{\mathbf
i}{n}\right)\frac{n!}{\mathbf k!(n-|\mathbf
i|)!(i_{1}-k_{1})!(i_{2}-k_{2})!}\\
&\times\mathbf x^{\mathbf e}(\mathbf x+\mathbf k \beta)^{\mathbf
k-\mathbf e}(\mathbf y-\mathbf x)^{\mathbf e}\left[(\mathbf
y-\mathbf x)+(\mathbf i-\mathbf k)\beta\right]^{\mathbf i-\mathbf
k-\mathbf e}\\
&\times\left(1-|\mathbf y|\right)\left[1-|\mathbf y|+(n-|\mathbf
i|)\beta\right]^{n-|\mathbf i|-1}\\
=&\left(1+n\beta\right)^{1-n}\sum^{n}_{k_{1}=0}\sum^{n-k_{1}}_{l_{1}=0}\sum^{n-k_{1}-l_{1}}_{k_{2}=0}\sum^{n-|\mathbf
k|-l_{1}}_{l_{2}=0}f\left(\frac{\mathbf k+\mathbf
l}{n}\right)\frac{n!}{\mathbf k!\mathbf l!(n-|\mathbf k|-|\mathbf
l|)!}\\
&\times\mathbf x^{\mathbf e}(\mathbf x+\mathbf k \beta)^{\mathbf
k-\mathbf e}(\mathbf y-\mathbf x)^{\mathbf e}(\mathbf y-\mathbf
x+\mathbf l\beta)^{\mathbf l-\mathbf
e}\\
&\times\left(1-|\mathbf y|\right)\left[1-|\mathbf y|+(n-|\mathbf
k|-|\mathbf l|)\beta\right]^{n-|\mathbf k|-|\mathbf l|-1}.
\end{split}
\end{equation}
On the other hand,
\begin{equation*}
\begin{split}
G^{\beta}_{n}(f;\mathbf
x)=&\left(1+n\beta\right)^{1-n}\sum^{n}_{k_{1}=0}\sum^{n-k_{1}}_{k_{2}=0}f\left(\frac{\mathbf
k}{n}\right)\binom{n}{\mathbf k}\mathbf x^{\mathbf e}(\mathbf
x+\mathbf k
\beta)^{\mathbf k-\mathbf e}\\
&\times\left(1-|\mathbf x|\right)\left[1-|\mathbf x|+(n-|\mathbf
k|)\beta\right]^{n-|\mathbf k|-1}.
\end{split}
\end{equation*}
In (\ref{eq:1.1}), if we put $y_{1}-x_{1}$, $1-y_{1}-x_{2}$ and
$n-|\mathbf k|$ in place of $u$, $v$ and $m$, respectively, one has
\begin{equation*}
\begin{split}
&\left(1-|\mathbf x|\right)\left[1-|\mathbf x|+(n-|\mathbf
k|)\beta\right]^{n-|\mathbf k|-1}\\=&\sum^{n-|\mathbf
k|}_{l_{1}=0}\binom{{n-|\mathbf
k|}}{l_{1}}(y_{1}-x_{1})(y_{1}-x_{1}+l_{1}\beta)^{l_{1}-1}(1-y_{1}-x_{2})\\
&\times\left[1-y_{1}-x_{2}+(n-|\mathbf
k|-l_{1})\beta\right]^{n-|\mathbf k|-l_{1}-1}.
\end{split}
\end{equation*}
Again in the equality (\ref{eq:1.1}), we replace $u$, $v$ and $m$ by
$y_{2}-x_{2}$, $1-|\mathbf y|$ and $m=n-|\mathbf k|-l_{1}$,
respectively, we find
\begin{equation*}
\begin{split}
&(1-y_{1}-x_{2})\left[1-y_{1}-x_{2}+(n-|\mathbf
k|-l_{1})\beta\right]^{n-|\mathbf k|-l_{1}-1}\\
=&\sum^{n-|\mathbf k|-l_{1}}_{l_{2}=0}\binom{{n-|\mathbf
k|-l_{1}}}{l_{2}}(y_{2}-x_{2})(y_{2}-x_{2}+l_{2}\beta)^{l_{2}-1}\\
&\times(1-|\mathbf y|)\left[1-|\mathbf y|+(n-|\mathbf k|-|\mathbf
l|)\beta\right]^{n-|\mathbf k|-|\mathbf l|-1}.
\end{split}
\end{equation*}
Making use of this in the above equality leads to
\begin{equation*}
\begin{split}
&\left(1-|\mathbf x|\right)\left[1-|\mathbf x|+(n-|\mathbf
k|)\beta\right]^{n-|\mathbf k|-1}\\=&\sum^{n-|\mathbf
k|}_{l_{1}=0}\sum^{n-|\mathbf k|-l_{1}}_{l_{2}=0}\binom{{n-|\mathbf
k|}}{\mathbf l}(\mathbf y-\mathbf x)^{\mathbf e}(\mathbf y-\mathbf
x+\mathbf
l\beta)^{\mathbf l-\mathbf e}\\
&\times\left(1-|\mathbf y|\right)\left[1-|\mathbf y|+(n-|\mathbf
k|-|\mathbf l|)\beta\right]^{n-|\mathbf k|-|\mathbf l|-1}.
\end{split}
\end{equation*}
Thus we conclude that
\begin{equation*}
\begin{split}
G^{\beta}_{n}(f;\mathbf
x)=&\left(1+n\beta\right)^{1-n}\sum^{n}_{k_{1}=0}\sum^{n-k_{1}}_{k_{2}=0}\sum^{n-|\mathbf
k|}_{l_{1}=0}\sum^{n-|\mathbf
k|-l_{1}}_{l_{2}=0}f\left(\frac{\mathbf
k}{n}\right)\frac{n!}{\mathbf k!\mathbf l!(n-|\mathbf k|-|\mathbf
l|)!} \\
&\times\mathbf x^{\mathbf e}(\mathbf x+\mathbf k \beta)^{\mathbf
k-\mathbf e}(\mathbf y-\mathbf x)^{\mathbf e} (\mathbf y-\mathbf
x+\mathbf
l\beta)^{\mathbf l-\mathbf e}\\
&\times \left(1-|\mathbf y|\right)\left[1-|\mathbf y|+(n-|\mathbf
k|-|\mathbf l|)\beta\right]^{n-|\mathbf k|-|\mathbf l|-1}.
\end{split}
\end{equation*}
Now changing the order of the two summations in the middle, we
obtain
\begin{equation}\label{eq:3.2}
\begin{split}
G^{\beta}_{n}(f;\mathbf x)
=&\left(1+n\beta\right)^{1-n}\sum^{n}_{k_{1}=0}\sum^{n-k_{1}}_{l_{1}=0}\sum^{n-k_{1}-l_{1}}_{k_{2}=0}\sum^{n-|\mathbf
k|-l_{1}}_{l_{2}=0}f\left(\frac{\mathbf
k}{n}\right)\frac{n!}{\mathbf k!\mathbf l!(n-|\mathbf k|-|\mathbf
l|)!}\\
&\times\mathbf x^{\mathbf e}(\mathbf x+\mathbf k \beta)^{\mathbf
k-\mathbf e}(\mathbf y-\mathbf x)^{\mathbf e}(\mathbf y-\mathbf
x+\mathbf l\beta)^{\mathbf l-\mathbf
e}\\
&\times\left(1-|\mathbf y|\right)\left[1-|\mathbf y|+(n-|\mathbf
k|-|\mathbf l|)\beta\right]^{n-|\mathbf k|-|\mathbf l|-1}.
\end{split}
\end{equation}
So, from (\ref{eq:3.1}) and (\ref{eq:3.2}) it follows that
\begin{equation*}
\begin{split}
G^{\beta}_{n}(f;\mathbf y)-G^{\beta}_{n}(f;\mathbf x)
=&\left(1+n\beta\right)^{1-n}\sum^{n}_{k_{1}=0}\sum^{n-k_{1}}_{l_{1}=0}\sum^{n-k_{1}-l_{1}}_{k_{2}=0}\sum^{n-|\mathbf
k|-l_{1}}_{l_{2}=0}\left[f\left(\frac{\mathbf k+\mathbf
l}{n}\right)-f\left(\frac{\mathbf
k}{n}\right)\right]\\
&\times\frac{n!}{\mathbf k!\mathbf l!(n-|\mathbf k|-|\mathbf l|)!}
\mathbf x^{\mathbf e}(\mathbf x+\mathbf k \beta)^{\mathbf k-\mathbf
e}(\mathbf y-\mathbf x)^{\mathbf e}(\mathbf y-\mathbf x+\mathbf
l\beta)^{\mathbf l-\mathbf
e}\\
&\times\left(1-|\mathbf y|\right)\left[1-|\mathbf y|+(n-|\mathbf
k|-|\mathbf l|)\beta\right]^{n-|\mathbf k|-|\mathbf l|-1}.
\end{split}
\end{equation*}
Now interchanging the order of the summations two times successively
and using the equality
$$\frac{n!}{\mathbf k!\mathbf l!(n-|\mathbf k|-|\mathbf l|)!}=\binom{{n}}{\mathbf
l}\binom{{n-|\mathbf l|}}{k_{1}}\binom{{n-k_{1}}-|\mathbf
l|}{k_{2}}$$
 we find
\begin{equation*}
\begin{split}
G^{\beta}_{n}(f;\mathbf y)-G^{\beta}_{n}(f;\mathbf x)
=&\left(1+n\beta\right)^{1-n}\sum^{n}_{l_{1}=0}\sum^{n-l_{1}}_{k_{1}=0}\sum^{n-k_{1}-l_{1}}_{l_{2}=0}\sum^{n-k_{1}-|\mathbf
l|}_{k_{2}=0}\left[f\left(\frac{\mathbf k+\mathbf
l}{n}\right)-f\left(\frac{\mathbf
k}{n}\right)\right]\\
&\times\frac{n!}{\mathbf k!\mathbf l!(n-|\mathbf k|-|\mathbf l|)!}
\mathbf x^{\mathbf e}(\mathbf x+\mathbf k \beta)^{\mathbf k-\mathbf
e}(\mathbf y-\mathbf x)^{\mathbf e}(\mathbf y-\mathbf x+\mathbf
l\beta)^{\mathbf l-\mathbf
e}\\
&\times\left(1-|\mathbf y|\right)\left[1-|\mathbf y|+(n-|\mathbf
k|-|\mathbf l|)\beta\right]^{n-|\mathbf k|-|\mathbf l|-1}\\
=&\left(1+n\beta\right)^{1-n}\sum^{n}_{l_{1}=0}\sum^{n-l_{1}}_{l_{2}=0}\sum^{n-|\mathbf
l|}_{k_{1}=0}\sum^{n-k_{1}-|\mathbf
l|}_{k_{2}=0}\left[f\left(\frac{\mathbf k+\mathbf
l}{n}\right)-f\left(\frac{\mathbf
k}{n}\right)\right]\\
&\times\frac{n!}{\mathbf k!\mathbf l!(n-|\mathbf k|-|\mathbf l|)!}
\mathbf x^{\mathbf e}(\mathbf x+\mathbf k \beta)^{\mathbf k-\mathbf
e}(\mathbf y-\mathbf x)^{\mathbf e}(\mathbf y-\mathbf x+\mathbf
l\beta)^{\mathbf l-\mathbf
e}\\
&\times\left(1-|\mathbf y|\right)\left[1-|\mathbf y|+(n-|\mathbf
k|-|\mathbf l|)\beta\right]^{n-|\mathbf k|-|\mathbf l|-1}\\
=&\left(1+n\beta\right)^{1-n}\sum^{n}_{l_{1}=0}\sum^{n-l_{1}}_{l_{2}=0}\left[f\left(\frac{\mathbf
k+\mathbf l}{n}\right)-f\left(\frac{\mathbf
k}{n}\right)\right]\binom{{n}}{\mathbf l}\\
&\times(\mathbf y-\mathbf x)^{\mathbf e}(\mathbf y-\mathbf x+\mathbf
l\beta)^{\mathbf l-\mathbf e}\sum^{n-|\mathbf
l|}_{k_{1}=0}\binom{{n-|\mathbf
l|}}{k_{1}}x_{1}(x_{1}+k_{1}\beta)^{k_{1}-1}\\
&\times\sum^{n-k_{1}-|\mathbf l|}_{k_{2}=0}\binom{{n-k_{1}}-|\mathbf
l|}{k_{2}}x_{2}(x_{2}+k_{2}\beta)^{k_{2}-1}\\
&\times\left(1-|\mathbf y|\right)\left[1-|\mathbf y|+(n-|\mathbf
k|-|\mathbf l|)\beta\right]^{n-|\mathbf k|-|\mathbf l|-1}.
\end{split}
\end{equation*}
Taking $u=x_{2}$, $v=1-|\mathbf y|$ and $m=n-k_{1}-|\mathbf l|$ in
(\ref{eq:1.1}), it is easily seen that
\begin{equation*}
\begin{split}
&(x_{2}+1-|\mathbf y|)\left[x_{2}+1-|\mathbf y|+(n-k_{1}-|\mathbf
l|)\beta\right]^{n-k_{1}-|\mathbf l|-1} \\
=&\sum^{n-k_{1}-|\mathbf l|}_{k_{2}=0}\binom{{n-k_{1}}-|\mathbf
l|}{k_{2}}x_{2}(x_{2}+k_{2}\beta)^{k_{2}-1} \left(1-|\mathbf
y|\right)\left[1-|\mathbf y|+(n-|\mathbf k|-|\mathbf
l|)\beta\right]^{n-|\mathbf k|-|\mathbf l|-1}.
\end{split}
\end{equation*}
Hence we can write
\begin{equation*}
\begin{split}
G^{\beta}_{n}(f;\mathbf y)-G^{\beta}_{n}(f;\mathbf x)
=&\left(1+n\beta\right)^{1-n}\sum^{n}_{l_{1}=0}\sum^{n-l_{1}}_{l_{2}=0}\left[f\left(\frac{\mathbf
k+\mathbf l}{n}\right)-f\left(\frac{\mathbf
k}{n}\right)\right]\binom{{n}}{\mathbf l}\\
&\times(\mathbf y-\mathbf x)^{\mathbf e}(\mathbf y-\mathbf x+\mathbf
l\beta)^{\mathbf l-\mathbf e}\sum^{n-|\mathbf
l|}_{k_{1}=0}\binom{{n-|\mathbf
l|}}{k_{1}}x_{1}(x_{1}+k_{1}\beta)^{k_{1}-1}\\
&\times(x_{2}+1-|\mathbf y|)\left[x_{2}+1-|\mathbf
y|+(n-k_{1}-|\mathbf l|)\beta\right]^{n-k_{1}-|\mathbf l|-1}.
\end{split}
\end{equation*}
Again in (\ref{eq:1.1}),  we replace $u$, $v$ and $m$ by $x_{1}$,
$x_{2}+1-|\mathbf y|$ and $m=n-|\mathbf l|$, respectively, to obtain
\begin{equation*}
\begin{split}
&(x_{1}+x_{2}+1-|\mathbf y|)\left[x_{1}+x_{2}+1-|\mathbf
y|+(n-|\mathbf
l|)\beta\right]^{n-|\mathbf l|-1} \\
=&(1-|\mathbf y-\mathbf x|)\left[1-|\mathbf y-\mathbf x|+(n-|\mathbf
l|)\beta\right]^{n-|\mathbf
l|-1}\\
=&\sum^{n-|\mathbf l|}_{k_{1}=0}\binom{{n-|\mathbf
l|}}{k_{1}}x_{1}(x_{1}+k_{1}\beta)^{k_{1}-1} (x_{2}+1-|\mathbf
y|)\\&\times\left[x_{2}+1-|\mathbf y|+(n-k_{1}-|\mathbf
l|)\beta\right]^{n-k_{1}-|\mathbf l|-1}.
\end{split}
\end{equation*}
This leads to
\begin{equation}\label{eq:3.3}
\begin{split}
G^{\beta}_{n}(f;\mathbf y)-G^{\beta}_{n}(f;\mathbf x)
=&\left(1+n\beta\right)^{1-n}\sum^{n}_{l_{1}=0}\sum^{n-l_{1}}_{l_{2}=0}\left[f\left(\frac{\mathbf
k+\mathbf l}{n}\right)-f\left(\frac{\mathbf
k}{n}\right)\right]\binom{{n}}{\mathbf l}\\
&\times(\mathbf y-\mathbf x)^{\mathbf e}(\mathbf y-\mathbf x+\mathbf
l\beta)^{\mathbf l-\mathbf e}(1-|\mathbf y-\mathbf
x|)\\
&\times\left[1-|\mathbf y-\mathbf x|+(n-|\mathbf
l|)\beta\right]^{n-|\mathbf l|-1}.
\end{split}
\end{equation}
Since $f\in Lip_{M}(\mu,S)$, we can get
\begin{equation}\label{eq:3.4}
\begin{split}
\left|G^{\beta}_{n}(f;\mathbf y)-G^{\beta}_{n}(f;\mathbf x)\right|
\leq &
M\left(1+n\beta\right)^{1-n}\sum^{n}_{l_{1}=0}\sum^{n-l_{1}}_{l_{2}=0}\left[\left(\frac{l_{1}}{n}\right)^{\mu}+\left(\frac{l_{2}}{n}\right)^{\mu}\right]\binom{{n}}{\mathbf l}\\
&\times(\mathbf y-\mathbf x)^{\mathbf e}(\mathbf y-\mathbf x+\mathbf
l\beta)^{\mathbf l-\mathbf e}(1-|\mathbf y-\mathbf
x|)\\
&\times\left[1-|\mathbf y-\mathbf x|+(n-|\mathbf
l|)\beta\right]^{n-|\mathbf l|-1}\\
=&M\left[G^{\beta}_{n}\left(t^{\mu}_{1};\mathbf y-\mathbf
x\right)+G^{\beta}_{n}\left(t^{\mu}_{2};\mathbf y-\mathbf
x\right)\right].
\end{split}
\end{equation}
Now consider the term $G^{\beta}_{n}\left(t^{\mu}_{1};\mathbf
y-\mathbf x\right)$. With the help of the equality
$\binom{n}{\mathbf l}=\binom{n}{l_{1}}\binom{n-l_{1}}{l_{2}}$ we can
write
\begin{equation*}
\begin{split}
G^{\beta}_{n}\left(t^{\mu}_{1};\mathbf y-\mathbf x\right)
=&\left(1+n\beta\right)^{1-n}\sum^{n}_{l_{1}=0}\sum^{n-l_{1}}_{l_{2}=0}\left(\frac{l_{1}}{n}\right)^{\mu}\binom{{n}}{\mathbf
l} (\mathbf y-\mathbf x)^{\mathbf e}(\mathbf y-\mathbf x+\mathbf
l\beta)^{\mathbf l-\mathbf e}\\&\times(1-|\mathbf y-\mathbf x|)
\left[1-|\mathbf y-\mathbf x|+(n-|\mathbf
l|)\beta\right]^{n-|\mathbf l|-1}\\
=&\left(1+n\beta\right)^{1-n}\sum^{n}_{l_{1}=0}\left(\frac{l_{1}}{n}\right)^{\mu}\binom{n}{l_{1}}(y_{1}-x_{1})
(y_{1}-x_{1}+l_{1}\beta)^{l_{1}-1}\\
&\times
\sum^{n-l_{1}}_{l_{2}=0}\binom{n-l_{1}}{l_{2}}(y_{2}-x_{2})(y_{2}-x_{2}+l_{2}\beta)^{l_{2}-1}\\
&\times(1-|\mathbf y-\mathbf x|) \left[1-|\mathbf y-\mathbf
x|+(n-|\mathbf l|)\beta\right]^{n-|\mathbf l|-1}.
\end{split}
\end{equation*}
In the equality (\ref{eq:1.1}), if we take $y_{2}-x_{2}$,
$1-|\mathbf y-\mathbf x|$ and $n-l_{1}$ in place of $u$, $v$ and
$m$, respectively, then we get
\begin{equation*}
\begin{split}
&\left[1-(y_{1}-x_{1})\right]\left[1-(y_{1}-x_{1})+(n-l_{1})\beta\right]^{n-l_{1}-1}\\
=&
\sum^{n-l_{1}}_{l_{2}=0}\binom{n-l_{1}}{l_{2}}(y_{2}-x_{2})(y_{2}-x_{2}+l_{2}\beta)^{l_{2}-1}
(1-|\mathbf y-\mathbf x|)\\
&\times \left[1-|\mathbf y-\mathbf x|+(n-|\mathbf
l|)\beta\right]^{n-|\mathbf l|-1}.
\end{split}
\end{equation*}
Therefore,
\begin{equation*}
\begin{split}
G^{\beta}_{n}\left(t^{\mu}_{1};\mathbf y-\mathbf x\right)
=&\left(1+n\beta\right)^{1-n}\sum^{n}_{l_{1}=0}\left(\frac{l_{1}}{n}\right)^{\mu}\binom{n}{l_{1}}(y_{1}-x_{1})
(y_{1}-x_{1}+l_{1}\beta)^{l_{1}-1}\\
&\times\left[1-(y_{1}-x_{1})\right]\left[1-(y_{1}-x_{1})+(n-l_{1})\beta\right]^{n-l_{1}-1}\\
=&Q^{\beta}_{n}\left(t_{1}^{\mu}; y_{1}-x_{1}\right).
\end{split}
\end{equation*}
Applying the H\"{o}lder inequality with conjugate pairs
$p=\frac{1}{\mu}$ and $q=\frac{1}{1-\mu}$, we find
\begin{equation*}
G^{\beta}_{n}\left(t^{\mu}_{1};\mathbf y-\mathbf
x\right)=Q^{\beta}_{n}\left(t_{1}^{\mu};
y_{1}-x_{1}\right)\leq\left[Q^{\beta}_{n}(t_{1};
y_{1}-x_{1})\right]^{\mu}\left[Q^{\beta}_{n}(1;
y_{1}-x_{1})\right]^{\mu}.
\end{equation*}
As mentioned before, since the univariate Cheney-Sharma operators
given by $Q^{\beta}_{n}$ reproduce constant and linear functions we
reach to
\begin{equation*}
G^{\beta}_{n}\left(t^{\mu}_{1};\mathbf y-\mathbf
x\right)\leq(y_{1}-x_{1})^{\mu}.
\end{equation*}

Similarly,
\begin{equation*}
G^{\beta}_{n}\left(t^{\mu}_{2};\mathbf y-\mathbf x\right)\leq
(y_{2}-x_{2})^{\mu}.
\end{equation*}
Thus from (\ref{eq:3.4}) it follows that
\begin{equation*}
\left|G^{\beta}_{n}(f;\mathbf y)-G^{\beta}_{n}(f;\mathbf x)\right|
\leq  M\left[(y_{1}-x_{1})^{\mu}+(y_{2}-x_{2})^{\mu}\right]
\end{equation*}
which implies that $G^{\beta}_{n}(f)\in Lip_{M}(\mu,S)$. In a
similar way the same result can be found for $\mathbf x\geq\mathbf
y$.  If $x_{1}\geq y_{1}$, $x_{2}\leq y_{2}$, then we obtain from
the above result for $(y_{1},x_{2})\in S$ that
\begin{equation*}
\begin{split} \left|G^{\beta}_{n}(f;\mathbf
y)-G^{\beta}_{n}(f;\mathbf
x)\right|\leq&\left|G^{\beta}_{n}\left(f;(x_{1},x_{2})\right)-G^{\beta}_{n}\left(f;(y_{1},x_{2})\right)\right|\\
&+
\left|G^{\beta}_{n}\left(f;(y_{1},y_{2})\right)-G^{\beta}_{n}\left(f;(y_{1},x_{2})\right)\right|\\
\leq &M\left[(y_{1}-x_{1})^{\mu}+(y_{2}-x_{2})^{\mu}\right]
\end{split}
\end{equation*}
Finally, for the case $x_{1}\leq y_{1}$, $x_{2}\geq y_{2}$ we have
the same result. This completes the proof.
\end{proof}
\begin{thm}
If $\omega$ is a  of modulus of continuity function, then
$G^{\beta}_{n}(\omega)$ is also a modulus of continuity function for
all $n \in\mathbb{N}$.
\end{thm}
\begin{proof}
Let $\mathbf x,\mathbf y\in S$ such that $\mathbf y\geq \mathbf x$.
Regarding $f$ as a modulus of continuity function $\omega$ in
(\ref{eq:3.3}) we have
\begin{equation*}
\begin{split}
G^{\beta}_{n}(\omega;\mathbf y)-G^{\beta}_{n}(\omega;\mathbf x)
=&\left(1+n\beta\right)^{1-n}\sum^{n}_{l_{1}=0}\sum^{n-l_{1}}_{l_{2}=0}\left[\omega\left(\frac{\mathbf
k+\mathbf l}{n}\right)-\omega\left(\frac{\mathbf
k}{n}\right)\right]\binom{{n}}{\mathbf l}\\
&\times(\mathbf y-\mathbf x)^{\mathbf e}(\mathbf y-\mathbf x+\mathbf
l\beta)^{\mathbf l-\mathbf e}(1-|\mathbf y-\mathbf
x|)\\
&\times\left[1-|\mathbf y-\mathbf x|+(n-|\mathbf
l|)\beta\right]^{n-|\mathbf l|-1}
\end{split}
\end{equation*}
which means that $G^{\beta}_{n}(\omega;\mathbf y)\geq
G^{\beta}_{n}(\omega;\mathbf x)$ when  $\mathbf y\geq \mathbf x$.\\
Moreover, from the property $(c)$ of modulus of continuity function
$\omega$, we can write
\begin{equation*}
\begin{split}
G^{\beta}_{n}(\omega;\mathbf y)-G^{\beta}_{n}(\omega;\mathbf x)
=&\left(1+n\beta\right)^{1-n}\sum^{n}_{l_{1}=0}\sum^{n-l_{1}}_{l_{2}=0}\omega\left(\frac{\mathbf
l}{n}\right)\binom{{n}}{\mathbf l}\\
&\times(\mathbf y-\mathbf x)^{\mathbf e}(\mathbf y-\mathbf x+\mathbf
l\beta)^{\mathbf l-\mathbf e}(1-|\mathbf y-\mathbf
x|)\\
&\times\left[1-|\mathbf y-\mathbf x|+(n-|\mathbf
l|)\beta\right]^{n-|\mathbf l|-1}\\
=&G^{\beta}_{n}(\omega;\mathbf y-\mathbf x).
\end{split}
\end{equation*}
This shows that $G^{\beta}_{n}(\omega)$ is semi-additive. Finally,
from the definition of $G^{\beta}_{n}$ it is obvious that
$G^{\beta}_{n}(\omega;\mathbf 0)=\omega(\mathbf 0)=0$. Therefore
$G^{\beta}_{n}(\omega)$ itself is a function of modulus of
continuity when $\omega$  is so.
\end{proof}

% ----------------------------------------------------------------

\begin{center}

\end{center}

% ----------------------------------------------------------------


\begin{thebibliography}{9999}

\bibitem{ar} O. Agratini, I. A. Rus, \textit{Iterates of a class of dicsrete linear operators via contraction principle},
Comment. Math. Univ. Carolin., \textbf{44}(2003), no.3, 555-563.

\bibitem{al} F. Altomare, M. Campiti, \textit{Korovkin-type approximaton theory and its applications}, Walter de Gruyter, Berlin-New York, 1994.

\bibitem{gy} G. Ba\c{s}canbaz-Tunca, Y. Tuncer,\textit{On a Chlodovsky variant of a multivariate beta operator}, J. Comput. Appl. Math., \textbf{235}(2011), no.16, 4816-4824.

\bibitem{gf} G. Ba\c{s}canbaz-Tunca, F. Ta\c{s}delen,\textit{On Chlodovsky form of the Meyer-König and Zeller operators}, An. Univ. Vest Timi\c{s}. Ser. Mat.-Inform.,
, \textbf{49}(2011), no.2, 137-144.

\bibitem{gha} G. Ba\c{s}canbaz-Tunca, H. G. \.{I}nce-\.{I}larslan, A. Eren\c{c}in, \textit{Bivariate Bernstein type operators}, Appl. Math. Comput., \textbf{273}(2016), 543-552.

\bibitem{gcs} G. Ba\c{s}canbaz-Tunca,  A. Eren\c{c}in, F. Ta\c{s}delen, \textit{Some properties of Bernstein type Cheney and Sharma operators}(submitted).

\bibitem{bep}  B.M. Brown, D. Elliott, D.F. Paget, \textit{Lipschitz constants for the Bernstein polynomials of a Lipschitz continuous function},
J. Approx. Theory, \textbf{49}(1987), no.2, 196-199.

\bibitem{fc}  F. Cao, \textit{Modulus of continuity, K-functional and Stancu operator on a simplex}, Indian J. Pure Appl. Math., \textbf{35}(2004), no.12, 1343-1364.

\bibitem{cao}  F. Cao, C. Ding, Z. Xu, \textit{On multivariate Baskakov operator}, J. Math. Anal. Appl., \textbf{307}(2005), no.1, 274-291.


\bibitem{ca}  T. C\v{a}tina\c{s}, D. Otrocol,  \textit{Iterates of multivariate Cheney-Sharma operators}, J. Comput. Anal. Appl., \textbf{15}(2013), no.7,
1240-1246.

\bibitem{cs} E. W. Cheney, A. Sharma, \textit{On a generalization of Bernstein polynomials}, Riv. Mat. Univ. Parma, \textbf{2}(5)(1964), 77-84.

\bibitem{c} M. Cr\v{a}ciun, \textit{Approximation operators constructed by means of Sheffer sequences}, Rev. Anal. Num\'{e}r. Th\'{e}or. Approx.,
 \textbf{30}(2001), no.2, 135-150.

\bibitem{df}  C. Ding, F. Cao, \textit{K-functionals and multivariate Bernstein polynomials}, J. Approx. Theory,
 \textbf{(155)}(2008), no.2, 125-135.


\bibitem{agf}  A. Eren\c{c}in, G. Ba\c{s}canbaz-Tunca, F. Ta\c{s}delen, \textit{Some preservation properties of MKZ-Stancu type operators}, Sarejevo J. Math.,
 10\textbf{(22)}(2014), no.1, 93-102.

\bibitem{agf1}  A. Eren\c{c}in, G. Ba\c{s}canbaz-Tunca, F. Ta\c{s}delen, \textit{Some  properties of the operators defined by Lupa\c{s}},
 Rev. Anal. Num\'{e}r. Th\'{e}or. Approx.,
 \textbf{43}(2014), no.2, 168-174.

\bibitem{fa}  M. D. Farca\c{s}, \textit{About Bernstein polynomials}, An. Univ. Craivo Ser. Mat. Inform.,
 \textbf{(35)}(2008), 117-121.

\bibitem{kh}  M. K. Khan, \textit{Approximation properties of Beta operators}, Progress in approximation theory, Academic Press, Boston, MA, (1991), 483-495.

\bibitem{khp}  M. K. Khan,  M. A. Peters, \textit{Lipschitz constants for some approximation operators of a Lipschitz continuous function}, J. Approx. Theory, \textbf{59}(1989), no.3, 307-315.


\bibitem{li}  Z. Li, \textit{Bernstein polynomials and modulus of continuity}, J. Approx. Theory, \textbf{102}(2000), no.1, 171-174.

\bibitem{lin} T. Lindvall, \textit{Bernstein polynomials and the law of large numbers}, Math. Sci., \textbf{7}(1982), no.2, 127-139.

\bibitem{sc} D. D. Stancu, C. Cisma\c{s}iu,\textit{On an
approximating linear positive operator of Cheney-Sharma}, Rev. Anal.
Num\'{e}r. Th\'{e}or. Approx.,
 \textbf{26}(1997), no.1-2, 221-227.

 \bibitem{slc} D. D. Stancu, L. A. C\v{a}bulea, D. Pop, \textit{Approximation of bivariate functions by means of the operators $S^{\alpha,\beta;a,b}_{m,n}$},
 Stud. Univ. Babe\c{s}-Bolyai Math., \textbf{47}(2002), no.4, 105-113.

 \bibitem{s} D. D. Stancu, \textit{ Use of an identity of A. Hurwitz for construction of a linear positive operator of approximation}, Rev. Anal. Num\'{e}r. Th\'{e}or. Approx.,
 \textbf{31}(2002), no.1, 115-118.

\bibitem{sst} D. D. Stancu, E. I. Stoica,  \textit{On the use Abel-Jensen type combinatorial formulas for construction and
investigation of some algebraic polynomial operators of
approximation}, Stud. Univ. Babe\c{s}-Bolyai Math.,
\textbf{54}(2009), no.4, 167-182.

\bibitem{es} E. I. Stoica,  \textit{ On the  combinatorial identities of Abel-Hurwitz type and their use in constructive theory of functions
}, Stud. Univ. Babe\c{s}-Bolyai Math., \textbf{55}(2010), no.4,
249-257.

\bibitem{t} I. Ta\c{s}cu,\textit{Approximation of bivariate functions by operators of Stancu-Hurwitz type}, Facta Univ. Ser. Math. Inform.,
(2005), no.20, 33-39.

\bibitem{ta} I. Ta\c{s}cu, A. Horvat-Marc, \textit{Construction  of Stancu-Hurwitz  operator for two variables}, Acta Univ. Apulensis Math. Inform.,
(2006), no.11, 97-101.

\bibitem{ti} T. Trif, \textit{An elementary proof of the preservation of Lipschitz constants by the Meyer-König and Zeller operators},
J. Inequal. Pure Appl. Math., \textbf{4}(2003), no.5, Article90,
3pp.

\end{thebibliography}
\end{document}